\documentclass[12pt]{article}
\usepackage{e-jc}

\usepackage{amsmath,graphicx}

\DeclareMathOperator{\ex}{ex}
\newcommand\ceil[1]{\lceil#1\rceil}

\dateline{May 29, 2020}{Oct 13, 2020}{TBD}


\Copyright{The authors. Released under the CC BY-ND license (International 4.0).}

\title{Sharper bounds and structural results for minimally nonlinear 0-1 matrices}

\author{Jesse Geneson\\
\small Department of Mathematics \& Statistics\\[-0.8ex]
\small San Jose State University \\[-0.8ex] 
\small San Jose, CA, U.S.A.\\
\small\tt jesse.geneson@sjsu.edu\\
\and
Shen-Fu Tsai\\
\small Google LLC\\[-0.8ex]
\small Kirkland, WA, U.S.A.\\
\small\tt parity@\{google,gmail\}.com}

\begin{document}

\maketitle


\begin{abstract}
The extremal function $ex(n, P)$ is the maximum possible number of ones in any 0-1 matrix with $n$ rows and $n$ columns that avoids $P$. A 0-1 matrix $P$ is called minimally nonlinear if $ex(n, P) = \omega(n)$ but $ex(n, P') = O(n)$ for every $P'$ that is contained in $P$ but not equal to $P$.

Bounds on the number of ones and the number of columns in a minimally nonlinear 0-1 matrix with $k$ rows were found in (CrowdMath, 2018). In this paper, we improve the upper bound on the number of ones in a minimally nonlinear 0-1 matrix with $k$ rows from $5k-3$ to $4k-4$. As a corollary, this improves the upper bound on the number of columns in a minimally nonlinear 0-1 matrix with $k$ rows from $4k-2$ to $4k-4$.

We also prove that there are not more than four ones in the top and bottom rows of a minimally nonlinear matrix and that there are not more than six ones in any other row of a minimally nonlinear matrix. Furthermore, we prove that if a minimally nonlinear 0-1 matrix has ones in the same row with exactly $d$ columns between them, then within these columns there are at most $2d-1$ rows above and $2d-1$ rows below with ones.
\end{abstract}

\section{Introduction}

The 0-1 matrix $M$ \emph{contains} 0-1 matrix $P$ if some submatrix of $M$ either equals $P$ or can be turned into $P$ by changing some ones to zeroes. Otherwise we say that $M$ \emph{avoids} $P$. The function $\ex(n, P)$ is defined as the maximum number of ones in any 0-1 matrix with $n$ rows and $n$ columns that avoids $P$. This function has been applied to many problems in combinatorics and discrete geometry, including the Stanley-Wilf conjecture \cite{MT}, the maximum number of unit distances in a convex $n$-gon \cite{Fure}, and robot navigation problems \cite{Mit}.

The 0-1 matrix extremal function has a linear lower bound of $n$ for all 0-1 matrices except those with all zeroes or just one entry. Call a 0-1 matrix $P$ \emph{linear} if $\ex(n, P) = O(n)$ and \emph{nonlinear} if $\ex(n, P) = \omega(n)$. F\"{u}redi and Hajnal's problem, which is only partially answered, is to find all linear 0-1 matrices $P$ \cite{FH}. Marcus and Tardos proved that every permutation matrix $P$ is linear \cite{MT}, solving the Stanley-Wilf conjecture, and this result was later extended to double permutation matrices \cite{G}.  

One way to approach F\"{u}redi and Hajnal's problem is to find all 0-1 matrices on the border of linearity. A 0-1 matrix $P$ is called \emph{minimally nonlinear} if $\ex(n, P) = \omega(n)$ but $\ex(n, P') = O(n)$ for every $P'$ that is contained in $P$ but not equal to $P$. By definition, every minimally nonlinear 0-1 matrix must avoid every other nonlinear 0-1 matrix.

Keszegh \cite{Ke} constructed a class $H_{k}$ of 0-1 matrices for which $\ex(n, H_{k}) = \Theta(n \log{n})$ and conjectured the existence of infinitely many minimally nonlinear 0-1 matrices contained in the class. This conjecture was confirmed in \cite{G}, without actually constructing an infinite family of minimally nonlinear 0-1 matrices.

CrowdMath \cite{crowdmath} proved that a minimally nonlinear 0-1 matrix with $k$ rows has at most $5k-3$ ones and $4k-2$ columns and bounded the number of minimally nonlinear 0-1 matrices with $k$ rows, also finding analogous results for ordered bipartite graphs and forbidden sequences. They posed the problems of finding the maximum number of columns in, the maximum number of ones in, and the total number of minimally nonlinear 0-1 matrices with $k$ rows.

In this paper, we sharpen all three bounds from the CrowdMath paper. In Section \ref{sb}, we prove that a minimally nonlinear 0-1 matrix with $k$ rows has at most $4k-4$ ones and $4k-4$ columns. As a corollary, this gives an improved bound on the number of minimally nonlinear 0-1 matrices with $k$ rows. In Section \ref{ss}, we prove several structural results about 0-1 matrices, including that there are not more than four ones in the top and bottom rows of a minimally nonlinear matrix and that there are not more than six ones in any other row of a minimally nonlinear matrix. We also prove that if a minimally nonlinear 0-1 matrix has ones in the same row with exactly $d$ columns between them, then within these columns there are at most $2d-1$ rows above and $2d-1$ rows below with ones.

\section{Sharper bounds}\label{sb}

Given a matrix $M$, we define a \emph{reflection} of $M$ as any matrix obtained by a horizontal or vertical flip of $M$, and we define a \emph{rotation} of $M$ as any matrix obtained by rotating $M$ by $\pi / 2$, $\pi$, or $3\pi / 2$. Moreover, we define a \emph{transformation} of $M$ as any matrix obtained by a rotation or a reflection of $M$. For many proofs in this paper, we will refer to the minimally nonlinear patterns $R$, $Q_{1}$, and $S_{1}$ below respectively~\cite{FH}.

$$
\begin{pmatrix}
\bullet & \bullet \\
\bullet & \bullet \\
\end{pmatrix}
\begin{pmatrix}
\bullet &         & \bullet \\
\bullet & \bullet \\
\end{pmatrix}
\begin{pmatrix}
\bullet &         & \bullet \\
        & \bullet &         & \bullet\\
\end{pmatrix}
$$

\begin{theorem}
Minimally nonlinear 0-1 matrices with $k$ rows have at most $4k-4$ ones.
\end{theorem}

In order to prove this theorem, we introduce some more terminology and prove an intermediate lemma. The statement holds for $S_1$, $Q_1$, $R$ and their transformations, so we define an \emph{exclusive minimally nonlinear matrix} to be a minimally nonlinear 0-1 matrix that is not equal to $S_1$, $Q_1$, $R$ or their transformations. It suffices to prove that every exclusive minimally nonlinear matrix with $k$ rows has at most $4k-4$ ones. 

Call a column \emph{singular} if it has a single one and any adjacent columns have a one in the same row. Call a 0-1 matrix $P$ \emph{potentially minimally nonlinear} if the following three conditions hold: (1) $P$ avoids $S_1$, $Q_1$, $R$, and all of their transformations, (2) $P$ has no singular column, and (3) $P$ has no empty column. The next lemma shows that it suffices to prove that every potentially minimally nonlinear 0-1 matrix with $k$ rows has at most $4k-4$ ones.

\begin{lemma}
Every exclusive minimally nonlinear 0-1 matrix $P$ is potentially minimally nonlinear.
\end{lemma}

\begin{proof}
$P$ must avoid $S_1$, $Q_1$, $R$, and all of their transformations, since $P$ is minimally nonlinear and $P$ is not equal to $S_1$, $Q_1$, $R$ or their transformations. Moreover $P$ has no singular column and no empty column, since $P$ is minimally nonlinear, and adding a singular column or an empty column to a linear matrix cannot change the extremal function to nonlinear, see Lemma 2.3(d,f,g) of \cite{gtar}. Thus $P$ is potentially minimally nonlinear.\end{proof}

\begin{theorem}
Every potentially minimally nonlinear 0-1 matrix with $k$ rows has at most $4k-4$ ones.
\end{theorem}

\begin{proof} 
We prove this by induction. For $k = 2$, note that any potentially minimally nonlinear matrix $P$ with two rows has at most one column with two ones since $P$ avoids $R$. If $P$ has a column with two ones, then that is the only column in $P$, or else $P$ would contain a singular or empty column or $Q_1$ or its transformation. If every column of $P$ has a single one, then $P$ has at most four total columns since $P$ has no singular or empty columns and $P$ avoids $S_1$ and its transformations.

For the inductive step, we introduce some terminology for relationships between rows in a potentially minimally nonlinear matrix $A$. We say that \emph{row $r$ encompasses row $s$} if row $r$ has a one somewhere to the left of the leftmost one of row $s$, and row $r$ also has a one somewhere to the right of the rightmost one of row $s$. Note that if row $r$ encompasses row $s$, then row $r$ has no ones between two ones of row $s$, or else $A$ would contain $S_1$ or its transformation. Note that if row $s$ has a one in a column between two ones of row $r$, then row $r$ must encompass row $s$ or else A would contain $S_1$, $Q_1$, or their transformations.

For the inductive hypothesis, suppose that every potentially minimally nonlinear 0-1 matrix with $k$ rows has at most $4k-4$ ones, and let $A$ be a potentially minimally nonlinear 0-1 matrix with $k+1$ rows. Let row $t$ denote the row that does not encompass any other row and has the rightmost one among such rows. If such rows are not unique then we pick one arbitrarily as row $t$.

Row $t$ must not have ones in non-adjacent columns, or else either row $t$ would encompass another row or $A$ would contain an empty column, singular column, $Q_1$, $S_1$, or their transformations. Hence row $t$ has at most two ones, and they are adjacent. If we remove row $t$ from $A$ to obtain $A'$, then $A'$ is not necessarily potentially minimally nonlinear since $A'$ could have empty or singular columns.

If $A'$ has a singular column, then it only has one singular column or else $A$ contains $R$. Moreover if $A'$ has a singular column $c$, then $A'$ has no empty column, since a singular column cannot be next to an empty column by definition. Removing $c$ from $A'$ produces a 0-1 matrix $A''$ with no empty columns and no singular columns, since any singular columns in $A''$ would have been singular in $A'$, but we removed the singular column from $A'$ to produce $A''$. Moreover $A''$ avoids $S_1$, $Q_1$, $R$, and all of their transformations. Thus $A''$ is potentially minimally nonlinear and has at most $4k-4$ ones by inductive hypothesis, so $A$ has at most $4k-1$ ones.

If $A'$ has no singular column, then it has at most two empty columns, which are adjacent. If we remove the empty columns, then we will produce at most two singular columns which are adjacent to the columns that we remove. We remove the empty columns and any resulting singular columns from $A'$ to produce $A''$ with no empty columns and no singular columns. $A''$ is potentially minimally nonlinear, so $A''$ has at most $4k-4$ ones by inductive hypothesis. We removed at most two resulting singular columns after we removed the empty columns from $A'$, and we deleted at most two ones when we removed row $t$, so $A$ has at most $4k$ ones.
\end{proof}

The fact that a minimally nonlinear 0-1 matrix with $k$ rows has at most $4k-4$ columns follows immediately from the last result.

\begin{corollary}
Minimally nonlinear 0-1 matrices with $k$ rows have at most $4k-4$ columns. 
\end{corollary}

\begin{proof}
This follows from the last result since a minimally nonlinear 0-1 matrix cannot have any empty columns, see \cite{Ke} and Lemma 2.3(f) of \cite{gtar}.
\end{proof}

We obtain a slight improvement on the upper bound for the number of minimally nonlinear 0-1 matrices with $k$ rows in \cite{crowdmath} using the last corollary with the same proof as in \cite{crowdmath}.

\begin{corollary}
For $k > 2$, the number of minimally nonlinear 0-1 matrices with $k$ rows is at most $\sum_{i = \ceil{(k+4)/4}}^{4k-4}(i^k - (i-1)^k)k^{i-1}$.
\end{corollary}

\section{Structural results}\label{ss}

In this section, we find bounds on the number of ones in each row of a minimally nonlinear 0-1 matrix, in addition to other structural results about minimally nonlinear 0-1 matrices. For a few of the results in this section, we will refer to the minimally nonlinear patterns $Q_{3}$ and $S_{2}$ respectively below~\cite{FH}.

$$
\begin{pmatrix}
\bullet &         & \bullet \\
        & \bullet &        \\
        &         & \bullet \\
\end{pmatrix}
\begin{pmatrix}
\bullet &         &         &         \\
        &         & \bullet &         \\
        & \bullet &         & \bullet\\
\end{pmatrix}
$$

We start with a surprising lemma about the structure of minimally nonlinear patterns that is based on $Q_{1}$, $S_{1}$, $S_{2}$, and their transformations.

\begin{lemma}
Cross lemma:
Suppose that $M$ is a minimally nonlinear 0-1 matrix. If $r_1<r_2<r_3$, $c_1<c_2<c_3$, and $M_{r_1,c_2}=M_{r_3,c_2}=M_{r_2,c_1}=M_{r_2,c_3}=1$, then $M_{x,y}=0$ if $x\notin(r_1,r_3)$ and $y\notin(c_1,c_3)$.
\end{lemma}
\begin{proof}
If not, $M$ contains $S_1$, $S_2$, $Q_1$, or one of their transformations.
\end{proof}

We introduce some notation for the remaining results in this section. Let $M$ be a matrix and $a$, $b$, $c$, $d$ be integers. Define $M_{a,[b,c]}$ to be the submatrix of $M$ obtained by restricting to row $a$ and columns $b$ through $c$. Define $M_{[a,b],c}$ to be the submatrix of $M$ obtained by restricting to rows $a$ through $b$ and column $c$. Finally, define $M_{[a,b],[c,d]}$  to be the submatrix of $M$ obtained by restricting to rows $a$ through $b$ and columns $c$ through $d$. The next two lemmas allow us to obtain bounds on the number of ones in each row of a minimally nonlinear 0-1 matrix.

\begin{lemma}
\label{lemma:two_and_one}
Suppose that $M$ is a minimally nonlinear $h\times w$ matrix not equal to $S_2$, $Q_3$, or their transformations. If $M_{h,c}=M_{h,d}=1$ and $d-c>1$, then submatrix $M_{1,[c+1,d-1]}$ contains a one.
\end{lemma}

\begin{proof}
Suppose not. Ones in $M_{[1,h],[c+1,d-1]}$ cannot be only in row $h$, as otherwise removing the columns that contain these ones still results in a nonlinear matrix, see Lemma 2.3(g) of \cite{gtar}, contradicting the minimal nonlinearity of $M$. Since $S_2$ and $Q_3$ are both nonlinear~\cite{FH}, $M$ avoids each of them and their transformations.

Let $x$ be the top one in $M_{[1,h],[c+1,d-1]}$. If there are multiple ones in the same row we arbitrarily pick one as $x$. Note that $M_{1,c}=M_{1,d}=0$ since otherwise $M$ contains a transformation of $Q_3$ on $\{x,M_{h,c},M_{h,d}\}$ and one of $M_{1, c}$ or $M_{1, d}$. Similarly in order for $M$ to avoid $S_2$ and its transformations, $M_{1,[d+1,w]}$ and $M_{1,[1,c-1]}$ must both be empty. Then the top row has no one, a contradiction.
\end{proof}

The same proof method can be used to show the following lemma.

\begin{lemma}
\label{lemma:two_in_row}
Suppose that $M$ is a minimally nonlinear $h\times w$ matrix not equal to $S_2$, $Q_3$, or their transformations. If $M_{r, c} = M_{r,d} = 1$ and $d - c > 1$ for some row $r$, then there are $2$ possibilities, at least one of which must be true:
\begin{enumerate}
\item There exists a row $r'$ above $r$ with some one between columns $c+1$ and $d-1$ inclusive, and every row above $r'$ also has some one between columns $c+1$ and $d-1$ inclusive.
\item There exists a row $r'$ below $r$ with some one between columns $c+1$ and $d-1$ inclusive, and every row below $r'$ also has some one between columns $c+1$ and $d-1$ inclusive.
\end{enumerate}
\end{lemma}

As a corollary, we can derive bounds on the number of ones in rows of a minimally nonlinear 0-1 matrix.

\begin{lemma}
\label{lemma:no_five_ones}
There cannot be more than four ones in the bottom (or top) row of a minimally nonlinear matrix $M$.
\end{lemma}
\begin{proof}
If there are at least five ones in the bottom row, let the first five be in columns $x<y<z<u<v$. By Lemma~\ref{lemma:two_and_one}, $M_{1,a}=M_{1,b}=1$ where $x<a<z$ and $z<b<v$, and $M$ contains $S_1$, a contradiction.
\end{proof}

\begin{lemma}
\label{lemma:no_seven_ones}
There cannot be more than six ones in any row of a minimally nonlinear matrix $M$.
\end{lemma}
\begin{proof}
If there are at least seven ones in some row of $M$ that is not the top or bottom, fix the row and let the first seven ones be in columns $x<y<z<u<v<w<t$. By Lemma~\ref{lemma:two_in_row}, $M$ has ones in columns $c, d, e$, each in the first or last row, with $x < c < z$, $z < d < v$, and $v < e < t$. Thus $M$ contains $S_1$, a contradiction.
\end{proof}

The next few lemmas allow us to describe the structure of the ones in the top and bottom rows of a minimally nonlinear 0-1 matrix. 

\begin{lemma}
\label{lemma:four_ones}
If a minimally nonlinear matrix $M$ has four ones in the bottom row, then they are in columns $x,x+1,y,y+1$ with $y>x+2$. Moreover the top row of $M$ has either one or two ones, and they lie in columns $[x+2,y-1]$. If the top row has two ones, then they are adjacent to each other.
\end{lemma}
\begin{proof}
By the same argument as in the proof of Lemma~\ref{lemma:no_five_ones}, no three ones in the bottom row of $M$ can be mutually non-adjacent. So the first one is adjacent to the second, and the third is adjacent to the fourth. If $y=x+2$, i.e. all four are adjacent, then by Lemma~\ref{lemma:two_and_one} we have $M_{1,x+1}=M_{1,y}=1$ and $M$ contains $R$.

If not all four are adjacent, i.e. $y>x+2$, then all the ones in the top row are in columns $[x+2,y-1]$, or else $M$ would contain $S_1$, $Q_1$, or their transformations. If any two of these ones are not adjacent, then by Lemma~\ref{lemma:two_and_one} a column between them has one at the bottom, a contradiction. 
\end{proof}

\begin{lemma}
If a minimally nonlinear matrix $M$ has exactly three ones in the bottom row, then either (1) they are in three consecutive columns and the top row only has a one in the middle column, or (2) two columns of them are adjacent, and the top row only has one or two adjacent ones in columns strictly between the two sets of columns. 
\end{lemma}
\begin{proof}
At least two of the three ones in the bottom row are adjacent, otherwise we would have three mutually non-adjacent ones. If the columns containing these ones are consecutive and numbered $x,x+1,x+2$, then by Lemma~\ref{lemma:two_and_one}, $M_{1,x+1}=1$ and the top row has no other one, or else $M$ would contain $R$, $Q_1$, or its transformation.

If only two of these columns are adjacent, say $x,x+1,y$ with $y>x+2$. Then the result follows from the same argument as in the proof of Lemma~\ref{lemma:four_ones}.
\end{proof}

Combining the preceding results, we are able to find all possibilities for the relative positions of the ones in the top and bottom rows of a minimally nonlinear 0-1 matrix. 

\begin{theorem}
If $M$ is a minimally nonlinear matrix, then the matrix obtained by restricting $M$ to the first and last rows and removing any empty columns must be one of the following, its reflection, or its rotation by $\pi$.
$$
\begin{pmatrix}
\bullet & \bullet \\
\bullet & \bullet \\
\end{pmatrix}
\begin{pmatrix}
\bullet &         & \bullet \\
\bullet & \bullet \\
\end{pmatrix}
\begin{pmatrix}
\bullet &         & \bullet \\
        & \bullet &         & \bullet\\
\end{pmatrix}
$$
$$
\begin{pmatrix}
        &         & \bullet &         &        \\
\bullet & \bullet &         & \bullet & \bullet\\
\end{pmatrix}
\begin{pmatrix}
        &         & \bullet & \bullet &         &        \\
\bullet & \bullet &         &         & \bullet & \bullet\\
\end{pmatrix}
$$
$$
\begin{pmatrix}
        & \bullet &         \\
\bullet & \bullet & \bullet \\
\end{pmatrix}
\begin{pmatrix}
        &         & \bullet &         \\
\bullet & \bullet &         & \bullet \\
\end{pmatrix}
\begin{pmatrix}
        &         & \bullet & \bullet &         \\
\bullet & \bullet &         &         & \bullet \\
\end{pmatrix}
$$
$$
\begin{pmatrix}
        & \bullet & \bullet &        \\
\bullet &         &         & \bullet\\
\end{pmatrix}
\begin{pmatrix}
        & \bullet &        \\
\bullet &         & \bullet\\
\end{pmatrix}
\begin{pmatrix}
        & \bullet \\
\bullet & \bullet \\
\end{pmatrix}
\begin{pmatrix}
        & \bullet & \bullet \\
\bullet & \bullet &         \\
\end{pmatrix}
\begin{pmatrix}
        &         & \bullet & \bullet \\
\bullet & \bullet &         &         \\
\end{pmatrix}
\begin{pmatrix}
        &         & \bullet \\
\bullet & \bullet &         \\
\end{pmatrix}
$$
$$
\begin{pmatrix}
        & \bullet \\ 
\bullet &         \\
\end{pmatrix}
\begin{pmatrix}
\bullet \\ 
\bullet \\
\end{pmatrix}
$$

\end{theorem}

Together, the next two results show that if a minimally nonlinear 0-1 matrix has ones in the same row with exactly $d$ columns between them, then within these columns there are at most $2d-1$ rows above and $2d-1$ rows below with ones.

\begin{lemma}
\label{lemma:two_and_two}
Minimally nonlinear matrices avoid the pattern $Y$ below. Note that the second row of $Y$ is empty.
$$
\begin{pmatrix}
         & \bullet &         \\
         &         &         \\
         & \bullet &         \\
 \bullet &         & \bullet \\
\end{pmatrix}
$$
\end{lemma}
\begin{proof}
Suppose for contradiction that a minimally nonlinear matrix $M$ contains $Y$. Let row $r$ be the row of $M$ that contains the second row of the copy of $Y$. Note that row $r$ can only contain at most a single one, and that one has to be in the same column as the top entry of the copy of $Y$, or else $M$ would contain $S_1$, $S_2$, or their transformations. Construct $M'$ by removing the row of $M$ that contains the second row of the copy of $Y$. We have $ex(n, M)\leq 2ex(n, M')$, see Lemma 2.3(g) of \cite{gtar}, so $M'$ is also nonlinear, a contradiction.
\end{proof}

\begin{theorem}
If two ones in a minimally nonlinear matrix $M$ are in the same row and have exactly $d>0$ other columns between them, then within these $d$ columns there are at most $2d-1$ rows above and $2d-1$ rows below with ones.
\end{theorem}
\begin{proof}
Consider the intersection of the $d$ columns and all the rows above. Each column has no more than two ones in the intersection according to Lemma~\ref{lemma:two_and_two}, and no column has non-adjacent ones in the intersection. If $2d$ rows are nonempty in the intersection, then every column in the intersection has exactly two ones in the intersection, and every nonempty row in the intersection has a single one in $M$, or else $M$ would contain $S_1$ or $Q_1$ or their transformations. Moreover every row in the intersection above the first nonempty row must be nonempty in the intersection, or else $M$ would contain $S_2$, $Q_3$, or their transformations since $M$ has no empty rows. Thus the top two rows in the intersection have ones in the same column and zeroes elsewhere, making $M$ still nonlinear when the top row is removed, see Lemma 2.3(d) of \cite{gtar}, a contradiction. Thus at most $2d-1$ rows in the intersection are nonempty.
\end{proof}

\section*{Acknowledgements}

The authors thank the anonymous referee for helpful comments and suggestions.



\begin{thebibliography}{1}

\bibitem{crowdmath}
PA~CrowdMath.
\newblock Bounds on parameters of minimally nonlinear patterns.
\newblock {\em The Electronic Journal of Combinatorics}, 25(1):1--5, 2018.

\bibitem{Fure}
Zolt{\'a}n F{\"u}redi.
\newblock The maximum number of unit distances in a convex n-gon.
\newblock {\em J. Comb. Theory, Ser. A}, 55(2):316--320, 1990.

\bibitem{FH}
Zolt{\'a}n F{\"u}redi and P{\'e}ter Hajnal.
\newblock Davenport-schinzel theory of matrices.
\newblock {\em Discrete Mathematics}, 103(3):233--251, 1992.

\bibitem{G}
Jesse~T Geneson.
\newblock Extremal functions of forbidden double permutation matrices.
\newblock {\em Journal of Combinatorial Theory, Series A}, 116(7):1235--1244,
  2009.

\bibitem{Ke}
Bal{\'a}zs Keszegh.
\newblock On linear forbidden submatrices.
\newblock {\em Journal of Combinatorial Theory, Series A}, 116(1):232--241,
  2009.

\bibitem{MT}
Adam Marcus and G{\'a}bor Tardos.
\newblock Excluded permutation matrices and the stanley--wilf conjecture.
\newblock {\em Journal of Combinatorial Theory, Series A}, 107(1):153--160,
  2004.

\bibitem{Mit}
J~Mitchell.
\newblock Shortest rectilinear paths among obstacles.
\newblock Technical report, Cornell University Operations Research and
  Industrial Engineering, 1987.

\bibitem{gtar}
G{\'a}bor Tardos.
\newblock On 0--1 matrices and small excluded submatrices.
\newblock {\em Journal of Combinatorial Theory, Series A}, 111(2):266--288,
  2005.

\end{thebibliography}

\end{document}